\documentclass[a4paper,12pt]{amsart}
\usepackage[utf8x]{inputenc}
\usepackage{amsmath}
\usepackage{amsfonts}
\usepackage{amssymb}
\usepackage{hyperref}
\usepackage{fourier}
\usepackage{epic}
\usepackage[a4paper,left=2.2cm,right=2.2cm,top=2.7cm, bottom=2.7cm]{geometry}

\title{Non-separating cocircuits and graphicness in matroids}
\author{Jo\~ao Paulo Costalonga}
\address{{\upshape joaocostalonga@yahoo.com.br }\\
	 Rua Bricio Mesquita, 17\\
	 Cachoeiro de Itapemirim, ES\\
	 29300-750\\
         Brazil.
	 }

\newenvironment{proofof}{}{\hfill$\Box$\medskip}

\newtheorem{lem}{Lemma}[section]
\newenvironment{lemma}{\begin{lem}\stepcounter{subsection}}{\end{lem}}
\newtheorem{theo}[lem]{Theorem}
\newenvironment{theorem}{\begin{theo}\stepcounter{subsection}}{\end{theo}}
\newtheorem{cor}[lem]{Corolary}
\newenvironment{corolary}{\begin{cor}\stepcounter{subsection}}{\end{cor}}
\newtheorem{notat}[lem]{Notation}

\newtheorem{prop}[lem]{Proposition}

\newtheorem{conj}[lem]{Conjecture}
\newenvironment{conjecture}{\begin{conj}\stepcounter{subsection}}{\end{conj}}

\newcommand{\CC}{\mathcal{C}}

\newcommand{\ncont}{\nsubseteq}
\newcommand{\cont}{\subseteq}
\renewcommand{\u}{\cup}
\renewcommand{\i}{\cap}
\newcommand{\s}{^*}
\newcommand{\del}{\backslash}
\newcommand{\emp}{\emptyset}
\newcommand{\flist}[2]{#1_{1},\dots,#1_{#2}}
\begin{document}

\begin{abstract}
Let $M$ be a $3$-connected binary matroid and let $Y(M)$ be the set of elements of $M$ avoiding at least $r(M)+1$ non-separating cocircuits of $M$. Lemos proved that $M$ is non-graphic if and only if $Y(M)\neq\emp$. We generalize this result when  by establishing that $Y(M)$ is very large when $M$ is non-graphic and $M$ has no $M\s(K_{3,3}''')$-minor if $M$ is regular. More precisely that $|E(M)-Y(M)|\le 1$ in this case. We conjecture that when $M$ is a regular matroid with an $M\s(K_{3,3})$-minor, then $r\s_M(E(M)-Y(M))\le 2$. The proof of such conjecture is reduced to a computational verification.
\end{abstract}
\maketitle

\section{Introduction}

A cocircuit in a connected matroid is said to be non-separating if its deletion results in a connected matroid. For a 3-connected graphic matroid, note that the non-separating cocircuits correspond to the stars of the vertices in its graphic representation.

Non-separating cocircuits play an important role in the understanding of the structure of graphic matroids, as we can see in some instances that follows. Non-separating cocircuits were first studied by Tutte~\cite{Draw}, in the cographic case, to give a characterization of the planar graphs. Tutte also proved that the non-separating cocircuits of the bond matroid of a 3-connected graph spans, over $GF(2)$, its cycle-space. Tutte's results was generalized by Bixby and Cunningham~\cite{Bixby}, as we summarize in the following theorem, which was conjectured by Edmonds:

\begin{theorem}\label{bixby}
Let $M$ be a $3$-connected binary matroid with at least $4$ elements. Then
\begin{enumerate}
\item [(a)] the non-separating cocircuits of $M$ span its cocircuit-space,
\item [(b)] each element of $M$ is in at least two non-separating cocircuits, and
\item [(c)] $M$ is graphic if and only if each element of $M$ is in at most two non-separating cocircuits.
\end{enumerate}
\end{theorem}

Lemos, in \cite{Lemos2004} and \cite{Lemos2009}, proved results of similar nature as synthesized in the next theorem:

\begin{theorem}\label{lemos}
Let $M$ be a $3$-connected binary matroid with at least $4$ elements. Then
\begin{enumerate}
\item [(a)] for $e\in E(M)$, the non-separating cocircuits of $M$ that avoid $e$ span a hyperplane in the cocircuit-space of $M$, in particular $e$ avoids at least $r(M)-1$ cocircuits of $M$. Moreover, $e$ avoids more than $r\s(M)-1$ non-separating cocircuits of $M$ if and only if the set of the non-separating cocircuits of $M$ avoiding $e$ is linearly dependent, and
\item [(b)] $M$ is graphic if and only if each element of $M$ avoids at most $r\s(M)-1$ non-separating cocircuits.
\end{enumerate}
\end{theorem}

There are another characterizations of graphicness in binary matroids using non-separating cocircuits in Kelmans~\cite{Kelmans-Graph}, Lemos, Reid, and Wu~\cite{Reid} and Mighton~\cite{Mighton}. Kelmans~\cite{Kelmans-vertex} gave a simple proof of Whitney's 2-isomorphism Theorem using non-separating cocircuits. Some algorithms for recognizing graphicness in binary matroids are based on concepts related to non-separating cocircuits, see Tutte~\cite{Tutte}, Cunningham~\cite{Cunningham}, Mighton~\cite{Mighton} and Wagner~\cite{Wagner}.
 
\newcommand{\y}{\widetilde{Y}}
Theorems \ref{bixby} and \ref{lemos} identify two sets of obstructions for graphicness in a $3$-connected binary matroid $M$. We define the set $X(M)$ as the set of elements of $M$ meeting more than two non-separating cocircuits and the set $Y(M)$ of the elements of $M$ avoiding more than $r\s(M)-1$ non-separating cocircuits. Equivalently, by Theorem \ref{lemos}, (a), we may define $Y(M)$ as the set of the elements of $M$ avoiding all the members of a linearly dependent family of non-separating cocircuits of $M$. Some sharp lower bounds for $|X(M)|$ and $|X(M)\i Y(M)|$  are given by Lai, Lemos, Reid and Shao~\cite{Lai} when $M$ is not graphic. In this work we prove that $Y(M)$ contains almost all elements of $E(M)$ if $M$ is not graphic. Define $\y(M):=E(M)-Y(M)$. The main result we establish here is:


\begin{theorem}\label{principal}
Let $M$ be a $3$-connected non-graphic binary matroid. Then
\begin{enumerate}
\item [(a)] if $M$ is non-regular, then $|\y(M)|\leq1$. Moreover, $\y(M)=\emp$ if $M$ has no $S_8$-minor or $M$ has an $PG(3,2)$-minor.
\item [(b)] if $M$ is regular with no $M\s(K_{3,3}''')$-minor or with a $M\s(K_5)$-minor, then $\y(M)=\emp$.
\end{enumerate}
\end{theorem}

The following conjecture generalizes this last theorem.

\begin{conjecture}\label{todo}
Let $M$ be a $3$-connected non-graphic binary matroid. Then $r\s_M(\y(M))\le2$.
\end{conjecture}

In this paper we reduce the proof of Theorem \ref{principal} and Conjecture \ref{todo}. The theoretical part of the proof of conjecture is included in this paper. The computational part is being prepared. The computational part of proof of Theorem \ref{principal} is ready, but not properly written yet. Because of these missing pieces this paper is a preliminary report. More precisely, the theoretical part of proof of Conjecture \ref{todo}, reduces its proof to the verification of the following:

\begin{conjecture}\label{todo2}
Let $M$ be a $3$-connected non-graphic regular matroid with an $M\s(K_{3,3}''')$-minor, with no $M\s(K_5)$-minor and satisfying $r\s(M)\le 9$. Then $r\s_N(\y(N))\le2$. 
\end{conjecture}

More precisely, we prove in this version of this paper:

\begin{theorem}\label{todo-teo}
Conjectures \ref{todo} and \ref{todo2} are equivalent.
\end{theorem}

Conjecture \ref{todo} yields the following generalization of Lemos' graphicness criterion:

\begin{conjecture}
Let $M$ be a $3$-connected binary matroid and $S\cont E(M)$ satisfy $r\s_M(S)\ge3$. Then the following assertions are equivalent:
\begin{enumerate}
\item [(a)] $M$ is graphic;
\item [(b)] Each element of $S$ avoids at most $r\s(M)-1$ non-separating cocircuits of $M$; and
\item [(c)] Each element of $S$ avoids no linearly dependent set of non-separating cocircuits of $M$.
\end{enumerate}
\end{conjecture}


\section{Some results in critical 3-connectivity}


Let $M$ and $N$ be $3$-connected matroids. We say that an element $e\in E(M)$ is {\it vertically $N$-removable} in $M$ if $co(M\del e)$ is a $3$-connected matroid with an $N$-minor.

We can summarize the dual versions of Lemma 3.4 and Theorem 3.1 of \cite{Whittle} and Theorem 1.3 of \cite{Costalonga2} in the following theorem:
\begin{theorem}\label{whittle}
Let $M$ and $N$ be $3$-connected matroids. For $k\in\{1,2,3\}$, if $M$ has an $N$-minor and $r\s(M)-r\s(N)\ge k$, then there is a $k$-coindependent set of $M$ whose elements are vertically $N$-removable in $M$.
\end{theorem}

If $M$ and $N$ are $3$-connected matroids, we say that a set $X\cont E(M)$ is {\it $N$-removable} if $M\del X$ is a $3$-connected matroid with an $N$-minor. We also say that an element $e\in E(M)$ is {\it $N$-removable} if $\{e\}$ is $N$-removable. Now, we define a special structure that will be largely used along this article. When $M$ is binary, we say that a list of distinct elements $e_1,e_2,e_3,f_1,f_2,f_3$ of $M$ is an {\it$N$-pyramid} with {\it top} $T\s:=\{e_1,e_2,e_3\}$ and {\it base} $T:=\{f_1,f_2,f_3\}$ provided:  
\begin{enumerate}
\item [(a)] $T\s$ is an $N$-removable triad of $M$,
\item [(b)] $T$ is a triangle of $M$,
\item [(c)] for $\{i,j,k\}=\{1,2,3\}$, $\{e_i,e_j,f_k\}$ is a triangle of $M$, and
\item [(d)] the elements $f_1$, $f_2$ and $f_3$ are $N$-removable in $M$. 
\end{enumerate}
In other words, such a list is an $N$-pyramid with top $T\s$ when the restriction of $M$ to its elements is isomorphic $M(K_4)$ as illustrated in the figure below, where $\{e_1,e_2,e_3\}$ is an $N$-removable triad of $M$ and $f_1$, $f_2$ and $f_3$ are also $N$-removable in $M$.
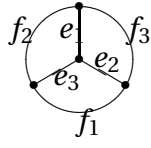
\begin{center}
\begin{picture}(64,64)
\put(32,32){\circle{72}}
\put(32,32){\circle*{3}}
\put(32,52){\circle*{3}}
\put(15,22){\circle*{3}}
\put(49.3,22){\circle*{3}}
\drawline[1](49.3,22)(32,32)(15,22)
\drawline[1](32,32)(32,52)
\put(24,40){$e_1$}
\put(37,29){$e_2$}
\put(22,22){$e_3$}
\put(30,5){$f_1$}
\put(5,40){$f_2$}
\put(49,40){$f_3$}
\end{picture}
\end{center}

Theorem \ref{whittle} does not hold for $k>3$, but, at the binary case, it can be extended by the following theorem, that is the dual version of Theorem 1.4, from \cite{Costalonga2}, written in terms of $N$-pyramids.

\begin{theorem}\label{costalonga}
Let $M$ be a $3$-connected binary matroid with a $3$-connected minor $N$ such that $r\s(M)-r\s(N)\ge 5$. Then
\begin{enumerate}
\item [(a)] $M$ has a $4$-coindependent set whose elements are vertically $N$-removable, or
\item [(b)] $M$ has an $N$-pyramid.
\end{enumerate}
\end{theorem}

\section{Lifting non-separating cocircuits from minors}

We define \newcommand{\R}{\mathcal{R}} $\R\s_A(M)$ as the set of non-separating cocircuits of $M$ avoiding $A$. We may write $R\s_e(M)$ instead of $R\s_{\{e\}}(M)$. We write $dep_A(M):=|\R\s_A(M)|-dim(\R\s_A(M))$, where $dim(\R\s_A(M))$ is the dimension of the space spanned by $\R\s_A(M)$ in the cocircuit space of $M$. We simplify the notation $dep_{\{e\}}(M)$ by $dep_e(M)$.

Observe that an element $e$ of a $3$-connected binary matroid $M$ is in $Y(M)$ if and only if $dep_e(M)>0$. The next result is Lemma 3.1 from \cite{Lemos2009}.

\begin{lemma}\label{lemmalemos}
Suppose that $e$ is an element of a $3$-connected binary matroid $M$ such that the cosimplification 
of $M\del e$ is $3$-connected. If $r\s ( M )\ge 4$, then it is possible to choose the ground set of $co( M \del e )$ so that, for each
$A \cont E (co( M \del e ))$,
\begin{equation}\label{eqlemos} dep_A ( M )\ge dep_{ A'} ( M ) \ge dep _A (co( M \del e )),\end{equation}
where $A'$ is the minimal subset of $E ( M )$ satisfying $A \cont A'$ and, for each triad $T\s$ of $M$ that meets both $e$ and $A$, $T\s − e \cont A'$.
\end{lemma}

Next we establish:

\begin{lemma}\label{ylift}
Suppose that $M$ is a $3$-connected binary matroid satisfying $r\s(M)\ge4$. If $e$ is an element of $M$ such that $co(M\del e)$ is $3$-connected, then we may choose the ground set of $co(M\del e)$ so that:
\begin{equation}\label{yeq}
\begin{array}{l}
Y(co(M\del e))\cont Y(M)\text{; and}\\
\{f\in E(M):\text{ there is }g \in Y(co(M\del e))\text{ with }\{f,g\}\in\CC\s(M\del e)\}\cont Y(M).
\end{array}
\end{equation}
\end{lemma}
\begin{proof}
Choose the ground set of $co(M\del e)$ satisfying condition \eqref{eqlemos} of Lemma \ref{lemmalemos}. Recall that $f\in Y(M)$ if and only if $dep_f(M)\ge1$. If $f\in Y(co(M\del e))$, then, by \eqref{eqlemos}, $dep_f(M)\ge dep_f(co(M\del e))\ge1$ and, therefore $f\in Y(M)$. Now, suppose that $\{e,f,g\}$ is a triad of $M$ and that $g$ in $Y(co(M\del e))$. By \eqref{eqlemos}, $dep_{\{f,g\}}(M)\geq dep_g(M)>0$. Thus $\R\s_{\{f,g\}}(M)$ is linearly dependent and, therefore, so is $\R\s_f(M)$, since it contains $\R\s_{\{f,g\}}(M)$. Thus $f\in Y(M)$.
\end{proof}

\begin{corolary}\label{ylift1}
Suppose that $M$ is $3$-connected binary matroid satisfying $r\s(M)\ge4$. If $e$ is an element of $M$ such that $co(M\del e)$ is $3$-connected, then we may choose the ground set of $co(M\del e)$ so that: 
\begin{equation}\label{yeq2}
\y(M)\cont cl\s_M( \y(co[M\del e])\u e ).                                                                                                                                                            \end{equation}
\end{corolary}

\begin{lemma}\label{ylift2}
Let $M$ be a $3$-connected binary matroid with an $N$-minor, satisfying $r\s(M)\ge4$. Suppose that $e_1,e_2,e_3,f_1,f_2,f_3$ is an $N$-pyramid of $M$, having top $T\s$ and base $T$. Denote $F:=T\u T\s$.
\begin{enumerate}
\item [(a)] If $D\s\in \R\s(M)$, then either $D\s=T\s$, $D\s\i F=\emp$, or $D\s\i F=\{e_i,f_j,f_k\}$, for some distinct elements $i,j,k\in\{1,2,3\}$.
\item [(b)] If $C\s\in \R\s(M\del T\s)$, then there is a unique $D\s\in\R\s(M)$ such that $C\s\cont D\s\subsetneq D\s\u T\s$. Moreover, either
\begin{enumerate}
 \item [(b1)] $C\s\i T=\emp$ and $D\s=C\s$, or
 \item [(b2)] there are distinct elements $i,j,k$ of $\{1,2,3\}$ such that $C\s\i F=\{f_i,f_j\}$ and $D\s=C\s\u e_k$.
\end{enumerate}
\item[(c)] $Y(M\del T\s)\cont Y(M)$. Moreover, if $f_i\in Y(M\del T\s)$, then $T_i\cont Y(M\del T\s)$ and $f\notin cl\s_M( \y(M))$.  
\item[(d)] If $r\s_{M\del T\s}(\y(M\del T\s))\le 2$ and, for each $i=\{1,2,3\}$, $r\s_{M\del f_i}(\y(M\del f_i))\le 2$, then $r\s_M(\y(M))\le 2$.  
\end{enumerate}
\end{lemma}
\begin{proof}
In this proof we set, for $\{i,j,k\}=\{1,2,3\}$, $T_k:=\{e_i,e_j,f_k\}$.

Let us prove (a). Suppose that $D\s\in \R\s(M)$, $D\s\neq T\s$ and $D\s\i F\neq\emp$. As $D\s$ intersects $F$, it follows that $D\s$ intersect $T_k$ for some $k$. By orthogonality with $T_k$, $D\s$ intersects $T\s$. Hence $|D\s\i T\s|=1$, since $M\del D\s$ is connected; say $e_1\in D\s$. By orthogonality with $T_2$ and $T_3$, it follows that $\{e_1,f_2,f_3\}\cont D\s$. As $M$ is binary, none of $T$, $T_2$ or $T_3$ is contained in $D\s$. It yields that $D\s\i F=\{e_1,f_2,f_3\}$. We proved (a).

Let us prove (b). First, we examine the case that $C\s\i T=\emp$. In this case, it is straight to check that $M\del C\s$ is connected. It is just left to show that $C\s$ is a cocircuit of $M$. Consider a cocircuit $D\s$ of $M$ such that $C\s\cont D\s\subsetneq C\s\u T\s$, say $e_1\notin D\s$. As $C\s\i T=\emp$, then $f_2,f_3\notin D\s$. By orthogonality with $T_2$ and $T_3$, $T\s\i D\s=\emp$. Thus $D\s=C\s$ and $C\s$ is a cocircuit of $M$. Moreover, in this case, (b1) holds. So we may assume that $C\s\i T\neq\emp$. By orthogonality with $T$, we may suppose that $C\s\i T=\{f_1,f_2\}$. Let $D_0\s$ be a cocircuit of $M$ such that $C\s\cont D\s_0\subsetneq C\s\u T\s$. By orthogonality with $T_1$ and $T_2$, either $D_0\s=C\s\u e_3$ or $D\s_0=C\s\u\{e_1,e_2\}$. Note that $C\s\u\{e_1,e_2\}=(C\s\u e_3)\Delta T\s$. Thus both $C\s\u e_3$ and $C\s\u\{e_1,e_2\}$ are cocircuits of $M$. But it is easy to check that that $M\del (C\s\u e_3)$ is connected and $M\del (C\s\u\{e_1,e_2\})$ is disconnected. Define $D\s:=C\s\u e_3$ to conclude (b) and (b2).

To prove (c), let $e\in Y(M\del T\s)$. As $\R\s_e(M)$ is linearly dependent, there are distinct non-separating cocircuits $\flist{C\s}{n}$ of $M\del T\s$ avoiding $e$ such that:
\begin{equation}\label{ylift2-eq1}
C\s_1\Delta\dots\Delta C\s_n=\emp.
\end{equation} 
For each $l=1,\dots,n$, define $D\s_l$ as the non-separating cocircuit of $M$ such that $C\s_l\cont D\s_l\cont C\s_l\u T\s$, as described in (b). Consider, for $\{i,j,k\}=\{1,2,3\}$, the following subsets of $\{1,\dots,n\}$:
\begin{eqnarray*}
B_i &:=& \{l:f_i\in C\s_l\} \,\text{, and}\\
A_{ij} &:=& \{l:f_i,f_j\in C\s_l\}=\{l:e_k\in D\s_l\}\,\text{ (this equality holds by (a)).}
\end{eqnarray*}

By \eqref{ylift2-eq1}, each $B_i$ has even cardinality. By (a), $B_i$ is equal to the disjoint union of $A_{ij}$ and $A_{ik}$. Thus $|A_{12}|,|A_{13}|$ and $|A_{23}|$ are congruent modulo $2$. Hence $D\s_1\Delta\dots\Delta D\s_n$ is equal to $\emp$ or $T\s$. Therefore $\flist{D\s}{n}$ or $\flist{D\s}{n}, T\s$ is a list of linearly dependent non-separating cocircuits of $M$ avoiding $e$. This proves the first part of (c).

For the second part of (c), say that $f_1\in Y(M\del T\s)$. Note that, as above, for $e=f_1$ we have $A_{12}=A_{13}=\emp$. Thus $|A_{23}|$ is even and, therefore, $D\s_1\Delta\dots\Delta D\s_n=\emp$. In particular this implies that $\flist{D\s}{n}$ is a list of linearly dependent cocircuits of $M$ avoiding $T_1$. To finish, note that, in this case, $\y(M)$ is contained in the cohyperplane $E(M)-T_1$ of $M$ and, therefore, $f_1\notin cl\s_M(\y(M))$. 


Now, we prove (d). Suppose, for a contradiction, that $r\s_M(\y(M))\ge 3$. As $M\del T\s$ is $3$-connected and binary and $T$ is a triangle of $M\del T\s$, and, by hypothesis, $r\s_{M\del T\s}(\y(M\del T\s))\le2$, it follows that $T\ncont \y(M\del T\s)$, say $f_1\notin \y(M\del T\s)$. Thus $f_1\in Y(M\del T\s)$. By (c), $f_1\notin \y(M)$. As, $M\del f_1$ is $3$-connected, by Corollary \ref{ylift1} for $e=f_1$, we have that $\y(M)\cont cl_M\s(\y(M\del f_1)\u f_1)$. But $r\s_M(\y(M\del f_1)\u f_1)\le 3$, because $r\s_{M\del f_1}(\y(M\del f_1))\le 2$. Thus, as $r\s_M(\y(M))\ge3$, $\y(M)$ spans $f_1$ in $M\s$. A contradiction to the second part of (c). This proves (d), and therefore, the lemma.
\end{proof}

The next lemma is a straight consequence of the submodularity of the rank function of a matroid.
\begin{lemma}\label{flats}
Let $M$ be a matroid, $\flist{X}{m}\cont E(M)$ and $n:=\displaystyle{\max_i}\,r_M(X_i)$. Suppose that $r_M(X_1\u\dots\u X_m)\ge n+1$. Then $r_M(X_1\i\dots\i X_m)\le n-1$.
\end{lemma}

\begin{lemma}\label{almost-engine}
Let $l\in\{0,1,2\}$ and $M$ and $N$ be $3$-connected binary matroids. Suppose that $M$ has an $N$-minor and $r\s(M)\ge 4$. If $M$ has a $(l+2)$-coindependent set $I\s$, such that for each $e\in I\s$, $e$ is vertically $N$-removable and $r\s_{co(M\del e)}[\y(co(M\del e))]\le l$, then $r\s_M(\y(M))\le l$.
\end{lemma}
\begin{proof}
Let $I\s:=\{e_1,\dots,e_{(l+2)}\}$. For $i=1,\dots,l+2$, we choose the ground set of $M_i:=co(M\del e_i)$ satisfying equation \eqref{yeq2}, in Corollary \ref{ylift1}. So we have that, for each $i$, $\y(M)\cont X_i:=cl\s_M(\y(M_i)\u e_i)$. Our hypothesis implies that, for each $i$, $r\s_M(X_i)\le l+1$. Define $X:=X_1\u\dots\u X_{(l+2)}$. As $I\cont X$ and $r\s_M(X)\ge l+2$ then, by the dual version of Lemma \ref{flats} for $n=l+1$, $r\s_M(X_1\i\dots\i X_m)\le l$. Thus $r\s(\y(M))\le l$, since $\y(M)\cont X_1\i\dots\i X_{l+2}$.
\end{proof}

\begin{lemma}\label{engine}
Let $M$ be a $3$-connected binary matroid with a $3$-connected minor $N$ with $r\s(M)\ge4$ and let $l\in\{0,1,2\}$. If $r\s(M)-r\s(N)\ge 2+l+\left\lfloor\frac{l}{2}\right\rfloor$ and $r_M(\y(M))\ge l+1$, then
\begin{enumerate}
\item [(a)] $M$ has a vertically $N$-removable element $e$ such that 
$\y(co[M\del e ])$ has rank at least $l+1$ in $co(M\del e)$, or
\item [(b)] $l=2$ and $M$ has an $N$-pyramid $e_1,e_2,e_3,f_1,f_2,f_3$ such that there is $K\in\{M\del \{e_1,e_2,e_3\}$, $M\del f_1$, $M\del f_2$, $M\del f_3\}$ satisfying $r_K(\y(K))\ge 3$.
\end{enumerate}
\end{lemma}
\begin{proof}
By Theorems \ref{whittle} and \ref{costalonga}, one of the two statements holds:
\begin{enumerate}
\item [(i)] $M$ has an $(l+2)$-coindependent set $I:=\{e_1,\dots,e_{(l+2)}\}$ whose elements are vertically $N$-removable; or
\item [(ii)] $l=2$ and $M$ has an $N$-pyramid $e_1,e_2,e_3,f_1,f_2,f_3$, with top $T\s$.
\end{enumerate}
By lemma \ref{almost-engine}, (i) implies (a). By lemma \ref{ylift2}, (d), (ii) implies (b). The lemma is proved.
\end{proof}

\section{Some initial cases}

\newcommand{\ktt}{
\put(2,0){\circle*{2}}
\put(8,0){\circle*{2}}
\put(14,0){\circle*{2}}
\put(2,6){\circle*{2}}
\put(8,6){\circle*{2}}
\put(14,6){\circle*{2}}}

\newcommand{\kttlines}{
\drawline(2,0)(2,6)(8,0)(8,6)(14,0)(14,6)(2,0)(8,6)
\drawline(2,6)(14,0)
\drawline(8,0)(14,6)
}

\newcommand{\kp}{\drawline(2,6)(8,6)}
\newcommand{\kpp}{\drawline(2,6)(8,6)(14,6)}
\newcommand{\kppp}{\drawline(2,6)(8,6)(14,6)\spline(14,6)(8,15)(2,6)}
\newcommand{\kq}{\drawline(2,0)(8,0)}
\newcommand{\kqq}{\drawline(2,0)(8,0)(14,0)}
\newcommand{\kqqq}{\drawline(2,0)(8,0)(14,0)\spline(14,0)(8,-1)(2,0)}
The proof of the next lemma is just a routine check.
\begin{lemma}\label{initial-cases}
If $M\in \{F_7, F_7\s, M\s(K_5), R_{10}\}$ then $Y(M)=E(M)$.
\end{lemma}
Consider the partition of the vertices of $K_{3,3,}$ into two stable sets $V_1$ and $V_2$. For $0\le j\le i\le 3$, we define $K_{3,3}^{(i,j)}$ by a simple graph obtained from $K_{3,3}$ by the addiction of $i$ edges joining vertices of $V_1$ and $j$ edges joining vertices of $V_2$. We also consider the already established notations $K_{3,3}':=K_{3,3}^{(1,0)}$, $K_{3,3}'':=K_{3,3}^{(2,0)}$ and  $K_{3,3}''':=K_{3,3}^{(3,0)}$.

We define a circuit $C$ in a connected matroid $M$ to be {\it non-separating} if $C$ is a non-separating cocircuit of $M\s$, that is, if $M/C$ is connected. We also say that a circuit $C$ of a $3$-connected graph $G$ is {\it non-separating}, if $E(C)$ is a non-separating circuit of $M(G)$.

\begin{lemma}\label{k33-ext}
Suppose that $M$ is a simple cographic matroid with an $M\s(K_{3,3})$-minor such that $r\s(M)=5$. Then
\begin{enumerate}
\item [(a)]If $M\ncong M\s(K_{3,3}''')$, then $Y(M)=E(M)$.
\item [(b)]If $M=M\s(K_{3,3}''')$, then $\y(M)$ is a triad of $M$ and $Y(M)=E(K_{3,3})$.
\item [(c)]$M$ has a $M\s(K_5)$-minor if and only if $M\cong M\s(K_{3,3}^{(i,j)})$ for some $i,j\in\{1,2,3\}$.
\end{enumerate}
\end{lemma}
\begin{proof}First note that $M\cong M\s(K_{3,3}^{(i,j)})$ for some $0\le j\le i\le3$.

Item (c) follows from the facts that $K_{3,3}'''$ has no $K_5$-minor and that $K_{3,3}^{(1,1)}$ has a $K_5$-minor. Indeed, $|si(K_{3,3}'''/e)|<10$ for all $e\in E(K_{3,3}''')$ and $si(K_{3,3}^{(1,1)}/f)\cong K_5$ where $f$ is the edge joining the two degree-$3$ vertices of $K_{3,3}^{(1,1)}$.


Let us verify (a). First suppose that $M\cong M\s(G)$ for some connected graph extending $K_{3,3}^{(1,1)}$. Consider the edge $f$ such that $co(M\del f)\cong M\s(K_5)$, as before. By Corollary \ref{ylift1} and Lemma \ref{initial-cases}, it follows that $\y(M)\cont\{f\}$. It is just left to show that $f$ avoids a linearly dependent set of non-separating circuits in $G$ to finish this case. In fact, note that the triangles of $G$ that does not contain the end-vertices of $f$ constitute such a set. Now, we verify (a) for the remaining graphs:

\begin{enumerate}
\item $K_{3,3}$: note that each $4$-circuit of $K_{3,3}$ is non-separating. Let $g\in E(K_{3,3})$ and $v$ an end-vertex of $g$. Note that the set of the $4$-cocircuits of $K_{3,3}$ avoiding $v$ is linearly dependent. Thus $g\in Y(M\s(K_{3,3}))$, and so $Y(M\s(K_{3,3}))=E(M\s(K_{3,3}))$.
\item $K_{3,3}'$,
  \begin{picture}(16,6)\ktt\kp\kttlines\end{picture}:
this graph has orbits 
  \begin{picture}(16,6)\ktt\kp\end{picture},
  \begin{picture}(16,6)\ktt\drawline(2,6)(8,0)(8,6)(14,0)(2,6)(2,0)(8,6)\end{picture} and
  \begin{picture}(16,6)\ktt\drawline(2,0)(14,6)(8,0)(14,6)(14,0)\end{picture} 
of the automorphism group of its bond matroid.
The set of representatives
  \begin{picture}(16,6)\ktt\drawline(2,0)(2,6)(8,6)\end{picture}
of the first two orbits avoid the list
  \begin{picture}(16,6)\ktt\drawline(2,0)(8,6)(8,0)(14,6)(2,0)\end{picture} , 
  \begin{picture}(16,6)\ktt\drawline(8,0)(8,6)(14,0)(14,6)(8,0)\end{picture} , 
  \begin{picture}(16,6)\ktt\drawline(2,0)(8,6)(14,0)(14,6)(2,0)\end{picture}
of linearly dependent non-separating circuits, while the element 
  \begin{picture}(16,6)\ktt\drawline(2,0)(14,6) \end{picture}
of the third orbit avoids
  \begin{picture}(16,6) \ktt\drawline(2,6)(8,0)(14,6)(14,0)(2,6) \end{picture} , 
  \begin{picture}(16,6) \ktt\drawline(2,6)(8,6)(8,0)(2,6)\end{picture} , 
  \begin{picture}(16,6)\ktt\drawline(2,6)(8,6)(14,0)(2,6)\end{picture} and 
  \begin{picture}(16,6)\ktt\drawline(8,0)(8,6)(14,0)(14,6)(8,0)\end{picture}.

\item $K_{3,3}''$, 
\begin{picture}(16,6) \ktt\kttlines\drawline(2,6)(8,6)(14,6) \end{picture}: 
analogously, we have orbits 
  \begin{picture}(16,6) \ktt\drawline(2,6)(8,6)(14,6) \end{picture} , 
  \begin{picture}(16,6) \ktt\drawline(2,0)(8,6)(8,0)(8,6)(14,0) \end{picture} and 
  \begin{picture}(16,6) \ktt\drawline(2,6)(8,0)(14,6)(2,0)(2,6)(14,0)(14,6) \end{picture}. 
The linearly dependent non-separating circuits
  \begin{picture}(16,6) \ktt\drawline(2,6)(14,0)(14,6)(2,0)(2,6) \end{picture} , 
  \begin{picture}(16,6) \ktt\drawline(2,6)(8,0)(14,6)(14,0)(2,6) \end{picture} and 
  \begin{picture}(16,6) \ktt\drawline(2,6)(8,0)(14,6)(2,0)(2,6) \end{picture} 
avoid the two first orbits, while the representative
  \begin{picture}(16,6) \ktt\drawline(2,6)(8,0) \end{picture}
of the third orbit avoids
  \begin{picture}(16,6) \ktt\drawline(2,6)(14,0)(14,6)(2,0)(2,6) \end{picture} , 
  \begin{picture}(16,6) \ktt\drawline(2,6)(2,0)(8,6)(2,6) \end{picture} , 
  \begin{picture}(16,6) \ktt\drawline(2,6)(14,0)(8,6)(2,6) \end{picture} , 
  \begin{picture}(16,6) \ktt\drawline(8,6)(14,6)(2,0)(8,6) \end{picture} and 
  \begin{picture}(16,6) \ktt\drawline(8,6)(14,6)(14,0)(8,6) \end{picture}.
\end{enumerate}
\newcommand{\anotheredge}{\drawline(2,6)(4,8)(7,9)(9,9)(12,8)(14,6)}
It is remaining to prove (b). Note that in $K_{3,3}'''$, \begin{picture}(16,6)
\ktt \kttlines \kpp \anotheredge
\end{picture} , the edge 
\begin{picture}(16,6)\ktt\anotheredge\end{picture} 
avoids exactly the following non-separating circuits:
  \begin{picture}(16,6) \ktt\drawline(2,0)(2,6)(8,6)(2,0)\end{picture} ,
  \begin{picture}(16,6) \ktt\drawline(8,0)(2,6)(8,6)(8,0)\end{picture} ,
  \begin{picture}(16,6) \ktt\drawline(14,0)(2,6)(8,6)(14,0)\end{picture} ,
  \begin{picture}(16,6) \ktt\drawline(2,0)(14,6)(8,6)(2,0)\end{picture} ,
  \begin{picture}(16,6) \ktt\drawline(8,0)(14,6)(8,6)(8,0)\end{picture} and
  \begin{picture}(16,6) \ktt\drawline(14,0)(14,6)(8,6)(14,0)\end{picture} , 
that constitute a linearly independent set. So, the orbit 
  \begin{picture}(16,6)\ktt\kpp\anotheredge\end{picture} 
is contained in $\y(M\s(K_{3,3}'''))$. But the representative 
  \begin{picture}(16,6) \ktt\drawline(14,0)(14,6)\end{picture} 
of the other orbit avoids the linearly dependent circuits
  \begin{picture}(16,6) \ktt\drawline(2,0)(2,6)(8,6)(2,0)\end{picture} ,
  \begin{picture}(16,6) \ktt\drawline(8,0)(2,6)(8,6)(8,0)\end{picture} ,
  \begin{picture}(16,6) \ktt\drawline(2,6)(2,0)(8,0)(2,6)\end{picture} and
  \begin{picture}(16,6) \ktt\drawline(8,6)(2,0)(8,0)(8,6)\end{picture} . 
This finishes the proof of (b) and of the lemma.
\end{proof}

\begin{lemma}\label{k5-start}
If $M$ is a $3$-connected regular matroid with an $M\s(K_5)$-minor and $r\s(M)\le 5$, then $M$ is cographic and $Y(M)=E(M)$.
\end{lemma}
\begin{proof}
First let us verify that $M$ is cographic. As $r(M)\le 5$, $M$ has no $R_{12}$-minor. If $M$ has an $R_{10}$-minor, as $R_{10}$ is a splitter for the class of the regular matroids, then $M\cong R_{10}$; a contradiction. Thus $M$ is cographic. If $r(M)=4$, then $M\cong M\s(K_5)$ and the lemma follows. So, we may assume that $r(M)=5$. As $M\s(K_5)$ is a splitter for the class of the cographic matroids with no $M\s(K_{3,3})$-minor, then $M$ is isomorphic to the bond matroid of a graph with 6 vertices extending $K_{3,3}$. The lemma follows from items (a) and (c) of Lemma \ref{k33-ext}.
\end{proof}

The next lemma has a computer assisted proof that will be approached in Section \ref{section-comput}.

\begin{lemma}\label{comput} Let $M$ be a $3$-connected binary matroid and $e\in E(M)$. 
\begin{enumerate}
\item [(a)]If $co(M\del e)\cong S_8$, then $|\y(M)|\le1$.
\item [(b)]If $r\s(M)=4$ and $M\ncong S_8$, then $Y(M)=E(M)$. Moreover $|\y(S_8)|=1$.
\item [(c)]If $co(M\del e)$ is isomorphic to $M\s(K_{3,3}^{(i,0)})$ for some $i\in\{0,1,2\}$, then $Y(M)=E(M)$.
\item [(d)]If $M$ has an element $e$ such that $co(M\del e)\cong PG(3,2)\s$ then $E(M)=Y(M)$.
\end{enumerate}
\end{lemma}

\section{Proof of the main theorem}

The statement of Theorem \ref{principal} just summarizes the lemmas proved in this section.

\begin{lemma}\label{k5-class}
If $M$ is a regular matroid with a $M\s(K_5)$-minor, then $E(M)=Y(M)$. 
\end{lemma}
\begin{proof}
Suppose the $M$ is a minimal counter-example to the lemma. By Lemma \ref{k5-start}, $r\s(M)\geq 6$. So, by Lemma \ref{engine}, for $l=0$, $M$ has a vertically $M\s(K_5)$-removable element $e$ such that $\y(co(M\del e))$ is non-empty; a contradiction to the minimality of $M$.
\end{proof}

\begin{lemma}\label{non-k333-class}
If $M$ is a non-graphic regular matroid with no $M\s(K_{3,3}''')$-minor then $E(M)=Y(M)$.
\end{lemma}
\begin{proof}
Suppose that $M$ is a minimal counter-example for the lemma. By lemma \ref{k5-class}, $M$ has no $M\s(K_5)$-minor. Thus, since $M$ is not graphic, $M$ has a $M\s(K_{3,3})$-minor. 

If $r\s(M)\ge 7$, then, by Lemma \ref{engine} for $l=0$, it follows that $M$ has a vertically $M\s(K_{3,3})$-removable element $e$ such that $\y(co(M\del e))\neq\emp$. But it contradicts the minimality of $M$. Hence $r\s(M)\le6$.

Note that $M$ is not isomorphic to $R_{10}$ nor have an $R_{10}$-minor, since $R_{10}$ is a splitter for the class of the regular
matroids. If $r\s(M)=6$, $M$ has a vertically $M\s(K_{3,3})$-removable element $e$. As $co(M\del e)$ has no minor isomorphic to $R_{10}$ or $R_{12}$, it follows that $co(M\del e)$ is cographic. In particular, $M$ is a corank-$5$ cographic matroid extending $M\s(K_{3,3})$. By Lemma \ref{k33-ext}(c), $co(M\del e)$ is isomorphic to $M\s(K_{3,3}^{(i,0)})$ for some $i\in\{0,1,2\}$. In this case the result follows from Lemma \ref{comput}(c). Thus $r\s(M)=5$. As before, $M$ has no $R_{10}$ or $R_{12}$-minor and $M$ is cographic. Now the result follows from Lemma \ref{k33-ext}, (a). 
\end{proof}

\begin{lemma}
Suppose that $M$ is a $3$-connected non-regular binary matroid. Then $|\y(M)|\le 1$. Moreover, if $M$ has no $S_8$-minor, then $Y(M)=E(M)$.
\end{lemma}
\begin{proof}
Note that $F_7$ satisfies the lemma. But $F_7$ and $F_7\s$ are the unique excluded minors for regularity in the class of the binary matroids. Moreover $F_7\s$ is a splitter for the class of binary matroids with no $F_7$-minor. Thus we may assume that $M$ has a $F_7\s$-minor.

First suppose that $M$ is a minimal counter-example for the first part of the lemma. By Lemma \ref{comput}(b), $r\s(M)\ge 5$. If $r\s(M)=5$ then, by Lemma \ref{engine} (for $l=0$), $M$ has an element $e$ such that $\y(co(M\del e))\neq\emp$. By Lemma \ref{comput}(b) again, $co(M\del e) \cong S_8$. But it contradicts Lemma \ref{comput}(a). Thus $r\s(M)\ge 6$. It follows by Lemma \ref{engine} (for $l=1$), that $M$ has a vertically $F_7\s$-removable element $e$ such that $|\y(co(M\del e)|\ge2$. Thus $co(M\del e)$ contradicts the minimality of $M$.

Now suppose that $M$ is a minimal counter-example for the second part of the lemma. 
If $r\s(M)\ge 5$, then, by Lemma \ref{engine}, for $l=0$, $M$ has a vertically $F_7\s$-removable element $e$ such that $\y(co(M\del e))\neq\emp$, a contradiction to the minimality of $M$. Thus $r\s(M)=4$. Now, the result follows from Lemma \ref{comput}(b).
\end{proof}

\begin{lemma}
If $M$ is a binary $3$-connected matroid with a $PG(3,2)\s$-minor, then $Y(M)=E(M)$.
\end{lemma}
\begin{proof}
Let $M$ be a minimal counter-example for the lemma. If $r\s(M)\ge 6$, then, by Lemma \ref{almost-engine} for $l=0$, we have a contradiction to the minimality $M$. So, $r\s(M)\le5$, but, by Lemma \ref{comput}(b), $r\s(M)\ge 5$. Thus $r\s(M)=5$. By Theorem \ref{whittle}, $M$ has a vertically $PG(3,2)\s$-removable element $e$. Thus $co(M\del e)\cong PG(3,2)\s$, since $PG(3,2)\s$ is a maximal rank-$4$ binary matroid. But it contradicts Lemma \ref{comput}(v).
\end{proof}

\begin{proofof}\emph{Proof of Theorem \ref{todo-teo}: }It is clear that Conjecture \ref{todo2} is a particular case of Conjecture \ref{todo}. In other hand if we consider a minimal counter-example $M$ for the converse, by Lemma \ref{engine} (for $l=2$), analogously to the preceding proofs, $M$ has a minor that contradicts the minimality of $M$.
\end{proofof}

\section{Extremal cases for the main theorem}
Here we denote by $Z_r$ the {\it binary rank-$r$ spike}: a matroid represented by a binary $(2r+1)\times r$ matrix in the form $[I_r|\bar{I_r}|\vec{1}]$, where $\bar{I_r}$ is $I_r$ with the symbols interchanged and $\vec{1}$ is a column full of ones. We use the respective labels $\flist{a}{r},\flist{b}{r},c$ in this representation. We also define, for $n\ge 4$, $S_{2n}:=Z_n\del b_n$.

\begin{prop}
For $n\ge 4$, $S_{2n}$ attains the bound $|\y(S_{2n})|=1$ in Theorem \ref{principal}.
\end{prop}
\begin{proof}
By theorem \ref{principal}(i), we have that $\y(S_{2n})\le 2$. So it is enough to verify that $\y(S_{2n})\neq\emp$. As $S_{2n}$ is self-dual we  prove the proposition by showing that $a_n$ avoids at most $n-1$ non-separating circuits of $S_{2n}$. Note that the spanning circuits of $S_{2n}$ are not non-separating. It is not hard to verify that the non-spanning circuits of $S_{2n}$ avoiding $a_n$ are those in the form $\{c, a_i, b_i\}$, for some $1\le i<n-1$, or in the form $\{a_i, b_i, a_j, b_j\}$, for some $1\le i<j<n-1$ (the reader may also see \cite{Oxley}, page 662). But $c$ is a loop of the matroids in the form $S_{2n}/\{a_i, b_i, a_j, b_j\}$, which are, therefore, disconnected. Thus $a_n$ avoids at most $n-1$ non-separating circuits.
\end{proof}

Let $n\ge3$ and let $V_1$ and $V_2$ be the members of a partition of the vertex set of $K_{3,n}$ into two stable sets, where $|V_1|=3$. Define $K_{3,n}'''$ as the graph obtained from $K_{3,n}$ by adding an edge joining each pair of vertices in $V_1$. Note that the unique non-separating cocircuits of $M\s(K_{3,n}''')$ are the triangles of $K_{3,n}'''$ meeting both $V_1$ and $V_2$. So, $\y(M\s(K_{3,n}))$ is the triad $E(K_{3,n}''')-E(K_{3,n})$. Thus we have an infinite set of matroids attaining the bound $r\s_M(\y(M))=2$ in Theorem \ref{principal} (b).

\section{Complementary matroids in relation to projective geometries and a handmade classification of the rank-4 3-connected binary matroids}\label{section-complementaries}

From the uniqueness of representability of binary and ternary matroids, and from \cite[6.3.15]{Oxley}, we may conclude that:

\begin{lemma}\label{aut-extension}
Let $q\in \{2,3\}$, $s\ge2$, and $X, Y\cont E(PG(s,q))$. Suppose that there is a matroid isomorphism $\varphi:PG(s,q)|X\rightarrow PG(s,q)|Y$. Then, there is an automorphism $\Phi$ of $PG(s,q)$ that extends $\varphi$ and which restriction to $E(PG(s,q))-X$ is a matroid isomorphism between $PG(s,q)\del X$ and $PG(s,q)\del Y$. 
\end{lemma}

If, for $q\in\{2,3\}$, $M$ is a rank-$r$ simple matroid representable over $GF(q)$ we have, for $s\geq r-1$, well defined up to isomorphisms, the {\it complementary of $M$ in relation to} $PG(s,q)$ as the matroid $PG(s,r)\del M:=PG(s,q)\del X$, where $X\cont E(PG(s,q))$ is a set that satisfies $M\cong PG(s,q)|X$. Lemma \ref{aut-extension} implies:

\begin{corolary}
Let $q\in\{2,3\}$,let $M$ and $N$ be simple rank-$r$ matroids representable over $GF(q)$ and let $s\ge max\{r(M),r(N)\}-1$. Then
\begin{enumerate}
\item [(a)]$M\cong N$ if, and only if, $PG(s,q)\del M\cong PG(s,q)\del N$; and:
\item [(b)]$M$ is isomorphic to a minor of $N$ if, and only if, $PG(s,q)\del N$ is isomorphic to a minor of $PG(s,q)\del M$.
\end{enumerate}
\end{corolary}

\renewcommand{\P}{\mathbb{P}}
\begin{theorem}\label{classification}
Let $\P:=PG(3,2)$. Up to isomorphisms, all the rank-$4$ binary $3$-connected matroids are:
\begin{enumerate}
\item [(i)] $F_7\s$, $S_8$, $AG(3,2)$ and $M(\mathcal{W}_4)$, up to 8 elements;
\item[(ii)] $Z_4\cong \P\del M(K_4)$, $P_9\cong\P\del [M(k_4-e)\oplus U_{1,1}]$, $M\s(K_{3,3})$ $\cong$ $\P\del [U_{2,3}\oplus U_{2,3}]$ and $M(K_5\del e)\cong\P\del P(U_{2,3},U_{3,4})$, with $9$ elements;
\item[(iii)] $\P\del M(K_4\del e)$, $\P\del [U_{2,3}\oplus U_{2,2}]$, $\P\del [U_{3,4}\oplus U_{1,1}]$ and $M(K_5)$ $\cong$ $\P\del U_{4,5}$, with $10$ elements;
\item [(iv)] $\P\del [U_{2,3}\oplus U_{1,1}]$, $\P\del U_{3,4}$ and $\P\del U_{4,4}$; with $11$ elements; and:
\item [(v)] $\P\del U_{1,1}$, $\P\del U_{2,2}$, $\P\del U_{2,3}$ and $\P\del U_{3,3}$ with more than $11$ elements.
\end{enumerate}
\end{theorem}

\begin{proof}
In this proof when we're talking about equality and uniqueness, it is up to isomorphisms. We're denoting by $M$ an arbitrary $3$-connected rank-$4$ binary matroid.

By the main result of Oxley \cite{Oxley87}, the only $3$-connected rank-$4$ binary matroids with no $M(\mathcal{W}_4)$-minor and are $F_7\s$, $AG(3,2)$, $S_8$, and $Z_4$. So, the rank-$4$ binary $3$-connected matroids up to $8$ elements are $F_7\s$, $AG(3,2)$, $S_8$ and $M(\mathcal{W}_4)$. Let us find the complementaries of these matroids in relation to $\P$.

It is easy to see that $AG(3,2)=\P\del F_7$. As the unique single-element deletion of $F_7$ is $M(K_4)$, it follows that the unique rank-$4$ simple binary single-element extension of $AG(3,2)$ is $\P\del M(K_4)=Z_4$. So, the complementary of $S_8=Z_4\del b_4$ is the unique single-element extension of $M(K_4)$ different from $F_7$, that is $M(K_4)\oplus U_{1,1}$. We also have that $M(\mathcal{W}_4)=\P\del M(\mathcal{W}_4\del b)$, were $b$ is an edge in the rim of $\mathcal{W}_4$, as shown in the  graphs and the respective matrices that represent then below.

\begin{center}{\small
\begin{picture}(40,40)(0,15)
\put(20,20){\circle{40}}
\put(20,20){\circle*{3}}
\put(20,40){\circle*{3}}
\put(20,0){\circle*{3}}
\put(0,20){\circle*{3}}
\put(40,20){\circle*{3}}
\drawline(20,0)(20,40)
\drawline(0,20)(40,20)
\put(28,21){$1$}
\put(21,6){$2$}
\put(8,13){$3$}
\put(15,26){$4$}
\put(36,3){$5$}
\put(0,3){$6$}
\put(0,32){$7$}
\put(36,32){$8$}
\end{picture}
$\bordermatrix{
&1&2&3&4&5&6&7&8\cr
&0&1&1&1&1&0&0&1\cr
&1&0&1&1&1&1&0&0\cr
&1&1&0&1&0&1&1&0\cr
&1&1&1&0&0&0&1&1
}\qquad$
\begin{picture}(40,20)(0,15)
\drawline(0,20)(20,0)(40,20)(20,40)
\put(20,20){\circle*{3}}
\put(20,40){\circle*{3}}
\put(20,0){\circle*{3}}
\put(0,20){\circle*{3}}
\put(40,20){\circle*{3}}
\drawline(20,0)(20,40)
\drawline(0,20)(40,20)
\put(26,20){$5$}
\put(15,9){$6$}
\put(5,20){$2$}
\put(15,28){$1$}
\put(31,5){$7$}
\put(5,5){$4$}
\put(30,30){$3$}
\end{picture}
$\bordermatrix{
&1&2&3&4&5&6&7\cr
&1&0&0&0&1&0&1\cr
&0&1&0&0&0&1&1\cr
&0&0&1&0&1&0&1\cr
&0&0&0&1&0&1&1
}$
}\\\end{center}

As the only excluded minors for graphicness in binary matroids are $F_7$, $F_7\s$, $M\s(K_5)$ and $M\s(K_{3,3})$, then all binary matroids up to $6$ elements are graphic. So, if $|E(M)|\geq9$, then $\P\del M$ is graphic.

The only possible degree sequences for simple connected graphs with $6$ edges and $4$ or $5$ vertices are: $(3,3,3,3)$, $(2,2,2,2,4)$, $(2,2,2,3,3)$ and $(1,2,2,3,4)$. Indeed, if $4$ appear twice in such a sequence, then the graph has at least $7$ edges. So $4$ appears at most once. The sum of the degrees must be $12$. So, as $4$ appears at most once, $1$ appears at most once too. Now it is easy to check that the possible sequence are those listed above. This implies that the unique simple matroids with $6$ elements and rank up to $4$ are: $M(K_4)$, $P(U_{2,3},U_{2,3})$, $M(K_{2,3})$, $U_{2,3}\oplus U_{3,4}$ and $M(K_4-e)\oplus U_{1,1}$.

Below we can see, in this order, a draw of $K_{3,3}$ and matrices representing $M(K_{2,3})$ and $\P\del M(K_{2,3})$: 
\begin{center}
\begin{picture}(40,20)(0,15)
\drawline(0,25)(20,0)(40,25)(20,50)(0,25)
\put(20,25){\circle*{3}}
\put(20,50){\circle*{3}}
\put(20,0){\circle*{3}}
\put(0,25){\circle*{3}}
\put(40,25){\circle*{3}}
\drawline(20,0)(20,50)
\put(20,10){$5$}
\put(31,7){$6$}
\put(5,7){$4$}
\put(20,32){$2$}
\put(31,36){$3$}
\put(5,36){$1$}
\end{picture}
{\small $\qquad
\bordermatrix{
&1&2&3&4&5&6\cr
&0&1&1&1&0&0\cr
&1&0&1&0&1&0\cr
&1&1&0&0&0&1\cr
&1&1&1&1&1&1}
\qquad\bordermatrix{
&1&2&3&4&5&6&7&8&9\cr
&1&0&0&0&1&1&0&1&1\cr
&0&1&0&0&1&0&1&1&1\cr
&0&0&1&0&0&1&1&1&1\cr
&0&0&0&1&0&0&0&0&1}$}\end{center}

Note that the last row in the second matrix corresponds to a $2$-cocircuit of $\P\del M(K_{2,3})$, which is not $3$-connected therefore. Note that all proper restrictions of $M(K_{2,3})$ are restrictions of $M(K_4)\oplus U_{1,1}$. So all $3$-connected rank-$4$ binary matroids with at least $9$ elements are complementaries of some restriction of $M(K_4)\oplus U_{1,1}$ or $M(\mathcal{W}_4\del b)$. In other hand, if $M$ is a restriction of $M(K_4)\oplus U_{1,1}$ or $M(\mathcal{W}_4\del b)$, then $\P\del M$ is a rank-$4$ extension of $S_8$ or $M(\mathcal{W}_4)$, and, therefore, $M$ is $3$-connected. This description corresponds to the matroids listed in the theorem.
\end{proof}

\section{Computational results}
\label{section-comput}

In this section we describe briefly the methods and procedures used to prove Lemma \ref{comput}. To get the list of non-separating cocircuits of a binary matroid $M$ and count how many of then avoid each element, we just use a brutal force algorithm that examines each linear combination of lines in a standard matrix representing $M$. The subroutines for this are based on well known algorithms. 

Consider a binary matrix $A$ with columns $\flist{c}{n}$, and $v=(\flist{v}{n})\in \{0,1,2\}^n$. Let, for each, $i=1,\dots,n$, $\bar{v_i}$ be the reminder of the division of $v_i$ by two. Moreover, let $i_1<\dots<i_k$ be the elements of $\{j: 1\le j\le n$ and $v_j=2\}$. We define
$$\Gamma(A,v):=\left(\begin{array}{c|c|c|c|c|c|c}
1       &\bar{v_1}&\dots &\bar{v_n}& 1      &\dots& 1      \\ 
\hline          
0  &c_1 &\dots & c_n&c_{i_1} &\dots&c_{i_k} \\                
\end{array}\right).$$

It is easy to check that, for binary matroids $M$ and $N$ and a binary matrix $A$, if $M\cong M[A]$ and there is $e\in E(N)$ such that $M\cong si(N/e)$, then there is $v\in\{0,1,2\}^n$ such that $N\cong M[\Gamma(A,v)]$. The definition of $\Gamma$ may look awkward at a first moment, but it is easy to deal computationally, and make it easier to enumerate such matroids $N$ as above. Another attractive property of $\Gamma$ is equation \eqref{relation} below.

For a binary matrix $A$ with columns labelled by $1,\dots,n$ and for an automorphism $\sigma$ of $M[A]$ it is straight to check that:
\begin{equation}\label{relation}
M[ \Gamma(A,(\flist{v}{n})) ] \cong M[ \Gamma(A,(v_{\sigma(1)},\dots,v_{\sigma(n)})) ].  
\end{equation}
The next lemma is a straight consequence of Bixby's Theorem about decomposition of non $3$-connected matroids into $2$-sums. 
\begin{lemma}
Let $N$ be a coloopless simple non-$3$-connected binary matroid with at least $4$ elements. Suppose that $e$ is an element of $N$ such that $si(N/e)$ is $3$-connected. Then $e$ belongs to a non-trivial series class of $N$.
\end{lemma}
From this lemma we conclude:
\begin{corolary}
Let $A$ be a binary matrix with $n\ge4$ columns and $v\in\{0,1,2\}^n$. If $M[A]$ is $3$-connected and $M[\Gamma(A,v)]$ is cossimple, then $M[\Gamma(A,v)]$ is $3$-connected.
\end{corolary}

\newcommand{\F}{\mathcal{F}}
\newcommand{\M}{\mathcal{M}}
\renewcommand{\L}{\mathcal{L}}\newcommand{\A}{\mathcal{A}}
Let $A:=[I_r|D]$ be an $r\times n$ binary matrix. We define: $\L(A):=\{\Gamma(A,v); v\in \{0,2\}^r\times\{0,1,2\}^{n-r}\}$. Let $\M$ be a family of binary matroids we define a family of binary matrices $\A$ to be a {\bf standard vector representation} of $\M$ if each matroid of $\M$ is isomorphic to $M[A]$ for some $A\in \A$ and all matrices in $\A$ are in the standard form. For a family of binary matrices $\A$ we denote $M[\A]:=\{M[A]: A\in \A\}$ and $M\s[\A]:=\{M\s[A]: A\in \A\}$ . We simplify the language saying that standard vector representation of $M[A]$ is a {\bf standard vector representation} of $\A$.

\begin{lemma} If $A$ is a binary matrix and $M$ is a matroid with an element $e$ such that $M[A]\cong co(M\del e)$ and $|E(M)|-|E(A(M))|\le k$, then  $M$ is isomorphic to a matroid represented by a matrix in $\L^k(A)$.
\end{lemma}

The following lemma describe the procedures that are being used to prove Conjecture \ref{todo2}.

\begin{lemma}\label{unique}
If $M$ is a $3$-connected regular matroid, with rank $5$, with a $M(K_{3,3}''')$-minor and with no $M(K_5)$-minor, then $M\cong M(K_{3,3}''')$.
\end{lemma}
\begin{proof}
Since, $M$ has no $R_{12}$ minor and $M\ncong R_{10}$ and $M$ is not cographic, them $M$ is graphic. So, if $M\ncong K_{3,3}'''$, then $M$ is the cycle matroid of a $6$-vertex simple graph properly extending $K_{3,3}'''$. The result follows from Lemma \ref{k33-ext}.
\end{proof}

\begin{lemma}\label{algorithm}
Suppose that $M'$ is a $3$-connected regular matroid with no $M\s(K_5)$ and with an $M\s(K_{3,3}''')$-minor such that $r\s(M')\le 9$ and $r\s_{M'}(\y(M'))\ge 3$. Then a matroid isomorphic to $M'$ can be found with the following procedure:
\end{lemma}
\begin{enumerate}
 \item Let $A_0$ be a standard binary matrix representing $M(K_{3,3}''')$. Let $\A_6$ be a standard vector representation of\\
 \centerline{$\{A\in \L{A_0}$, $M\s[A]$ is regular,  $3$-connected and has no $K_5$-minor and  $\y(M\s[A])\neq \emp\}$.}
 Check if all the matroids $M\in M\s[\A_6]$ satisfy $\y(M\s)\le 2$.
 
\item Let $\L_6=\A_6$. For $i=7,8,9$ do the following step;

\item Let $$\L_i:=\bigcup\limits_{A\in \L_{i-1}}\L(A)$$ and let $\A_i$ be a set of representatives of the following family:\\
\centerline{$\{A\in \L_i, M[A]$ $M$ is regular, $3$-connected and has no $K_5$-minor and $r\s\y(M\s[A])\ge 2\}$.}
Check if all the matroids $M\in M\s[\A_i]$ satisfy $\y(M\s)\le 2$.
\end{enumerate}

\begin{proofof}\emph{Proof of Lemma \ref{algorithm} }: We have to prove that if such $M'$ exists, then $M^{'*}$ is isomorphic to a matroid in $M\s[\A_6]\u \cdots\u M\s[\A_9]$. Suppose for a contradiction that this does not hold.

Let $r=r\s(M')$. By Lemma \ref{whittle} and by Lemma \ref{unique}, there is a chain of matroids $M(K_{3,3}''')=M_5,M_6,\dots M_r=M^{'*}$,  such that for each $i=6,\dots,r$ there is an element $e_i\in E(M_i)$ such that $M_{i-1}=si(M_i/e_i)$.


If $r=6$ it is clear that there is a matroid in $M\s[\A_6]$ isomorphic do $M$. So $r\ge 7$.

If $r=7$, since $r(M')-r(M(K_{3,3}'''))\ge 2$, by Lemma \ref{engine} for $l=0$, $r_{M_6}(\y(M_6\s))\ge 1$. So $M^{'*}$ is isomorphic to a matroid in $\A_7$. Hence $r\ge 8$.

If $r=8$, since $r(M')-r(M(K_{3,3}'''))\ge 3$, by Lemma \ref{engine}, for $l=1$, $r_{M_7}(\y(M_7\s))\ge 2$. So $M^{'*}$ is isomorphic to a matroid in $\A_8$. Hence $r=9$. Analogously we prove that $M'\in\A_9$ and arrive at a contradiction.
\end{proofof}

\section*{Acknowledgments}
The author thanks to Manoel Lemos for suggesting an approach for one of his conjectures, which generalization gave rise to the results established here.

\end{document}